\documentclass[12pt]{article}

\title{A Note on Twisted Bernoulli Measures}
\author{Altan Erdo\u{g}an\footnote{Department of Mathematics, Gebze Technical University}}

\usepackage{amssymb,amsmath,amsthm,amsfonts,mathrsfs,amscd,setspace}
\usepackage{fancyhdr}
\newtheorem*{defn}{Definition}
\newtheorem*{rmk}{Remark}
\newtheorem{thm}{Theorem}
\newtheorem{cor}{Corollary}
\newtheorem{ex}{Example}
\newtheorem{lemma}{Lemma}
\newtheorem{prop}{Proposition}
\numberwithin{equation}{section}

\begin{document}
\date{}
\maketitle{}

\begin{abstract}
We introduce the twisted Bernoulli measures as a family of $p$-adic measures parametrized by the complement of the open disc with radius 1 and centered at 1 in the completion of an algebraic closure of p-adic numbers. These measures are the higher order versions of the measure used by Koblitz and Coleman to interpret ($p$-adic) polylogarithms. We also prove that these measures are the unique $p$-adic measures that can be obtained from polynomials over the field $\mathbb{Q}(y)$ which is similar to the uniqueness property of Bernoulli measures.
\end{abstract}
\newpage

\section{Introduction}
The Bernoulli polynomials $B_k(x)$, $ k \in \mathbb{Z}_{\geq 0}$ are defined by
\[
\dfrac{te^{xt}}{e^t - 1} = \sum_{k=0}^{\infty} \dfrac{B_k(x) t^k}{k!}.
\]
$B_k = B_k(0)$ is called the $k$-th Bernoulli number. Bernoulli polynomials have important applications in number theory and related fields. Among them a significant appereance of Bernoulli polynomials occur in $p$-adic analysis via Bernoulli distributions relating them to ($p$-adic) $L$-functions defined by Leopoldt and Kubota \cite{Leo_Kub}. Convenient references are \cite{Koblitz_book} and \cite{Mazur}. 

The Bernoulli distributions are defined as $\mu_{B,k} (a + (p^N)) = p^{N(k-1)} B_k(a/p^N)$ on the compact open subsets of the form $a + (p^N)$ for $0 \leq a \leq p^N - 1$ in $\mathbb{Z}_p$ and is extended to $\mathbb{Z}_p$ linearly. 

Let $\mu_{B,k}$ and $\mu_{k,\alpha}$ denote the Bernoulli distributions and the Bernoulli measures (i.e. the  regularized Bernoulli distributions) respectively for $k \in \mathbb{Z}_{\geq 1}$ and $\alpha\in \mathbb{Z}_p^*$. Then the $p$-adic zeta function is defined by one of the following expressions,
\[
\zeta_p(1-k) := \dfrac{1}{\alpha^{-k} - 1} \int\limits_{\mathbb{Z}_p^*} x^{k-1} d\mu_{1,\alpha} = \dfrac{1}{k(\alpha^{-k} - 1)}  \int\limits_{\mathbb{Z}_p^*} 1 d\mu_{k,\alpha} = (1-p^{k-1})(-B_k/k)
\]
for $k \in \mathbb{Z}_{\geq 1}$ and extended to $\mathbb{Z}_p$ by $p$-adic interpolation (Note that $\alpha \in \mathbb{Z}_p^{*}$ is only used to regularize the distributions, so these expressions are actually independent of $\alpha$).

In a similar way we can define the twisted Bernoulli polynomials $\beta_k(x,y)$ by
\[
\dfrac{te^{xt}}{y e^t - 1} = \sum_{k=0}^{\infty} \dfrac{\beta_k(x,y) t^k}{k!}.
\]
where $y \neq 1$. Note that $\beta_k(x,y)$ is a polynomial in $x$ over $\mathbb{Z}[y,1/(y-1)]$. The $k$-th twisted Bernoulli number is defined as $\beta_k(0,y)$. The basic properties of the twisted Bernoulli polynomials and relations of them with Lerch zeta function can be found in \cite{apostol}. These basic properties reveal the combinatorial similarities between the classical and twisted Bernoulli polynomials. In this paper we show that twisted Bernoulli polynomials can also be used to construct $p$-adic measures which are naturally related to ($p$-adic) polylogarithms, and this relation is similar to the one between Bernoulli measures and the ($p$-adic) zeta functions   

Throughout the paper $\mathbb{C}_p$ will denote the completion of a fixed algebraic closure of $\mathbb{Q}_p$. For $z \in \mathbb{C}$, $|z|$ denotes the Euclidean norm and for $\alpha \in \mathbb{C}_p$, $|\alpha|_p$ denotes the normalized norm on $\mathbb{C}_p$ where $|p|_p = 1/p$. 

The polylogarithms are defined as 
\[
Li_{s}(z) = \sum_{n=1}^{\infty} \dfrac{z^n}{n^s}
\]
for $|z-1| < 1$, $s \in \mathbb{Z}$ and extended by analytic continuation in a multi-valued way.  A more precise (and useful) definition can be made through a system of differential equations, but for the purposes of this paper we may skip this definition. The reader may refer to \cite{coleman}. Note that $Li_{s}(z)$ can also be seen as a $p$-adic function which is similarly convergent for $|z-1|_p <1$. But in this case it can not be extended to a larger domain in $\mathbb{C}_p$. Consequently the $p$-adic polylogarithms $Li^{(p)}_{s}(z)$ may be introduced as
\[
Li^{(p)}_{s}(z) = \sum_{n=1, p \nmid n}^{\infty} \dfrac{z^n}{n^s}= Li_{s}(z) - p^{-s} Li_{s}(z^p)  ,\,\,\,\, |z-1|_p <1.
\]
Now $Li^{(p)}_{s}(z)$ can be extended to a larger domain, $|z-1|_p > 1/p^{p-1}$  in $\mathbb{C}_p$. Again the reader is consulted to \cite{coleman} in advance for a complete exposition.

The $p$-adic polylogarithms can also be defined in terms of $p$-adic integrals as we outline below. The details of the below arguments can be found in \cite{koblitz_book1} and \cite{coleman}.

Let $\mu_z$ be the measure defined as $\mu_{z}(a+ (p^N)) = z^a / (1- z^{p^N})$ for $|z-1|_p \geq 1$. 
Then for $k \in \mathbb{Z}_{\geq 1}$ and $|z-1|_p \geq 1$, we have  
\[
Li_{1-k}(z) = \int\limits_{\mathbb{Z}_p} x^{k-1} d\mu_{z} 
\]
(Chapter II of \cite{koblitz_book1}). Note that for $k \in \mathbb{Z}_{\geq 0}$, $Li_{-k}(z) \in \mathbb{Q}(z)$. So we can see $Li_{-k}(z)$ as a function on any extension of $\mathbb{Q}$, and in particular on $\mathbb{C}$ or $\mathbb{C}_p$. A similar equality also holds for $p$-adic polylogartihms due to Coleman (Lemma 7.2 of \cite{coleman}); for any $k \in \mathbb{Z}$ and $|z-1|_p \geq 1$ we have 
\[
Li^{(p)}_{1-k}(z) = \int\limits_{\mathbb{Z}_p^*} x^{k-1} d\mu_{z}.
\]

Another remarkable equality for $Li_{1-k}(z)$ is obtained by Apostol in \cite{apostol}; for $|z| = 1$, $z \neq 1$, we have
\[
Li_{1-k}(z) = -\beta_k(0,z) / k.
\]
Indeed Apostol proves a more general result for some special values of Lerch zeta functions, but we consider only a special case here.

Combining these results we obtain
\[
Li_{1-k}(z) = \int\limits_{\mathbb{Z}_p} x^{k-1} d\mu_{z} = -\beta_k(0,z) / k
\]
for $k \in \mathbb{Z}_{\geq 1}$ and for an algebraic number $z$ with  $|z| = 1$ and $|z-1|_p = 1$.  

The motivation for this paper is the equalities, 
\begin{align}
\label{padicversion}
Li^{(p)}_{1-k}(z) = \int\limits_{\mathbb{Z}^{*}_p} x^{k-1} d\mu_{z} = - \dfrac{\beta_k(0,z) - p^{k-1} \beta_k(0,z^p)}{k}
\end{align}
which is obtained from the previous one by restricting the domain to $\mathbb{Z}_p^{*}$. Note that the first equality is valid for any $k \in \mathbb{Z}$, but the second equality is only valid for $k \in \mathbb{Z}_{\geq 1}$. So we have an analogue of the equalities 
\[
\zeta_p(1-k)=\dfrac{1}{\alpha^{-k} - 1} \int\limits_{\mathbb{Z}_p^*} x^{k-1} d\mu_{1,\alpha} = \dfrac{1}{k(\alpha^{-k} - 1)}  \int\limits_{\mathbb{Z}_p^*} 1 d\mu_{k,\alpha} = (1-p^{k-1})(-B_k/k)
\]
except the term involving an integral with respect to the measure $\mu_{k,\alpha}$. 

In this paper we aim to locate this missing integral in (\ref{padicversion}). Explicitly we will construct a family of $p$-adic measures $\mu_{\beta,k,z}$ for $ k \in \mathbb{Z}_{\geq 1}$ using twisted Bernoulli polynomials where the measure $\mu_z$ defined by Koblitz appears as $\mu_{z} = - \mu_{\beta,1,z}$ in this family and 
\[
\int\limits_{{\mathbb{Z}}^{*}_p} x^{k-1} d\mu_{\beta,1,z} = (1/k) \int\limits_{\mathbb{Z}^{*}_p} 1 d\mu_{\beta,k,z}.
\]
As a result we will also see that (\ref{padicversion}) holds for any $z \in \mathbb{C}_p$ with $|z-1|_p \geq 1$. These measures are complete analogue of Bernoulli measures in the sense that $\mu_{\beta,k,z}$ plays the role of $d(x^k)$ for a fixed $z$. 

We proceed as follows. First we prove the uniqueness of such measures (up to a multiple) under certain conditions and give general results for some integrals. Then we show the existence of these measures and also relate them to Bernoulli measures. 

\newpage

\section{A family of $p$-adic measures parametrized by a subset of $\mathbb{C}_p$}
\label{general}

In this section we seek $p$-adic measures $\mu_{k,z}$, $k \in \mathbb{Z}_{\geq 1}$ defined as 
\[
\mu_{k,z}(a+(p^N)) =p^{N(k-1)} z^a f_k(a/p^N, z^{p^N}) 
\] 
for some functions $f_k(x,y) \in \mathbb{Q}(y)[x]$. We will prove the uniqueness of such measures up to a multiple depending only on $k$. We will also see that the analogy between $\mu_{k,z}$ and $d(x^k)$ is a natural outcome of the uniqueness assertion. This has been implemented in Theorem \ref{indexshift} where some integrals with respect to $\mu_{k,z}$ for various $k$ are computed. Consequently the reader may drop the index $k$ in advance to simplify the notation until Theorem \ref{indexshift}.

We start with a basic result which can easily be adopted for any distribution. Indeed it is a general property of distributions. 

\begin{lemma}
\label{fdecomposition}
Let $z \in C_p$. Suppose that for $k \in \mathbb{Z}_{\geq 1}$ there exists a rational function $f_k(x,y) \in \mathbb{Q}(x,y)$ such that the map defined on compact open subsets of $\mathbb{Z}_p$ of the form $a+(p^N)$, $0 \leq a \leq p^N -1$ as
\[
\mu_{k,z}(a + (p^N)) = p^{N(k-1)} z^a f_k(a/p^N, z^{p^N})
\]
extends to a $p$-adic distribution on $\mathbb{Z}_p$. Then for any $n \geq 1$, 
\[
f_k(x,y) = p^{n(k-1)} \sum_{b=0}^{p^n-1} y^{b} f_k((x+b)/p^n, y^{p^n}).
\]
\end{lemma}

\begin{proof}
The case $n=1$ follows by definition of $\mu_{k,z}$ and the equality 
\[
\mu(a+(p^N)) = \sum_{b=0}^{p-1} \mu(a+bp^N+ (p^{N+1})).
\]
Now the claim 
\[
f(x,y) = p^{n(k-1)} \sum_{b=0}^{p^n-1} y^{b} f((x+b)/p^n, y^{p^n}) \nonumber 
\]
follows by induction on $n$. 
\end{proof}

The converse trivially holds, i.e. if 
\[
f_k(x,y) = p^{n(k-1)} \sum_{b=0}^{p^n-1} y^{b} f_k((x+b)/p^n, y^{p^n})
\] 
for all $n \geq 1$ (or for only $n=1$) then $\mu_{k,z}$ extends to a $p$-adic distribution. We will use Lemma \ref{fdecomposition} in the proof of the following theorem which will play an important role in the uniqueness assertion. 

\begin{thm}
\label{prep}
Let $z \in C_p$. Suppose that for $k \in \mathbb{Z}_{\geq 1}$ there exists a rational function $f_k(x,y) \in \mathbb{Q}(x,y)$ such that the map defined on compact open subsets of $\mathbb{Z}_p$ of the form $a+(p^N)$, $0 \leq a \leq p^N -1$ as
\[
\mu_{k,z}(a + (p^N)) = p^{N(k-1)} z^a f_k(a/p^N, z^{p^N})
\]
extends to a $p$-adic distribution on $\mathbb{Z}_p$. Suppose that $f_k(x,y)$ and partial derivatives of $f_k(x,y)$ with respect to $x$ is defined at $(0,0)$. Then
\[
y f_k(x+1,y) - f_k(x,y) = h(k) x^{k-1}
\]
where $h(k) \in \mathbb{Q}$ only depends on $k$. 
\end{thm}
\begin{proof}
Let $g_k(x,y) = y f_k(x+1,y) - f_k(x,y)$. Once $k$ is fixed we may simply write $f(x,y)$ and $g(x,y)$. Now by Lemma \ref{fdecomposition} for any $n \geq 1$ we have
\[
y f(x+1,y) = p^{n(k-1)} \sum_{b=0}^{p^n-1} y^{b+1} f((x+1+b)/p^n, y^{p^n}) = p^{n(k-1)} \sum_{b=1}^{p^n} y^{b} f((x+b)/p^n, y^{p^n})
\]
which implies that
\begin{align*}
g(x,y) &= y f(x+1,y) - f(x,y) \\ 
&= y p^{n(k-1)} \sum_{b=0}^{p^n-1} y^{b} f((x+1+b)/p^n, y^{p^n}) - p^{n(k-1)} \sum_{b=0}^{p^n-1} y^{b} f((x+b)/p^n, y^{p^n}) \\
&= p^{n(k-1)} \sum_{b=1}^{p^n} y^{b} f((x+b)/p^n, y^{p^n}) - \sum_{b=0}^{p^n-1} y^{b} f((x+b)/p^n, y^{p^n}) \\
&= p^{n(k-1)} (y^{p^n} f(1 + x/p^n, y^{p^n}) -  f(x/p^n, y^{p^n})) \\
&= p^{n(k-1)} g(x/p^n, y^{p^n})
\end{align*}

Let $G(x,y) = \frac{(\partial g)^{k-1}}{\partial x^{k-1}}(x,y)$. Taking the derivative of the equality $g(x,y) =p^{n(k-1)} g(x/p^n, y^{p^n})$ with respect to $x$ , $k-1$ times, we obtain
\[
G(x,y) = G(x/p^n, y^{p^n})
\]
Now for any $(x,y) \in {\mathbb{Q}}^2$ in the domain of $G(x,y)$ with $|y| < 1$ we have that
\[
G(x,y) = G(x/p^n, y^{p^n}) = \lim_{n \rightarrow \infty} G(x/p^n, y^{p^n}) = G(0,0)
\]
which implies that $G(x,y)$ is constant and so $g(x,y)$ is indeed a polynomial in $x$ of the form
\[
g(x,y) = \sum_{i=1}^{k} g_i(y) x^{k-i}
\]
for some rational functions $g_i(y) \in \mathbb{Q}(y)$. Now we will show that $g_i(y)$ are independent of $y$. Since $g(x,y) = p^{n(k-1)} g(x/p^n, y^{p^n})$, we have that
\[
 \sum_{i=1}^{k} g_i(y) x^{k-i} = p^{n(k-1)} \sum_{i=1}^{k} g_i(y^{p^n}) x^{k-i} 1/p^{n(k-i)} = \sum_{i=1}^{k} g_i(y^{p^n}) x^{k-i} p^{n(i-1)}.
\]
Equating the coefficients we obtain 
\[
g_i(y) = g_i(y^{p^n}) p^{n(i-1)} 
\]
for $i=1,2,...,k$. Now let $y \in \mathbb{Q}$ with $|y|_p <1$, so that
\[
\lim_{n \rightarrow \infty} y^{p^n} = 0 
\] 
where the limit is taken in $\mathbb{Q}_p$. So we have that 
\[
g_i(y) = \lim_{n \rightarrow \infty} g_i(y^{p^n}) p^{n(i-1)} = 0 \mbox{  if i } \neq 1 \mbox{ and } g_1(y) = g_1(0).
\]
Setting $g_1(0) = h(k)$ we obtain $g(x,y) = \sum_{i=1}^{k} g_i(y) x^{k-i} = g_1(0) x^{k-1} = h(k) x^{k-1}$ as desired (Recall that $g(x,y)$ and so $g_i$ depend on $k$). 
\end{proof}

The independence of $y f_k(x+1,y) - f_k(x,y)$ from $y$ is a major restriction on $f_k(x,y)$. Indeed we will see that additionally if $f_k(x,y) \in \mathbb{Q}(y)[x]$ then $f_k(x,y)$ is uniquely determined by $h(k)$. We will need the following lemma for this. 

\begin{lemma}
\label{coefficientrecursion}
In addition to the hypotheses of Theorem \ref{prep}, also assume that $f_k(x,y) \in \mathbb{Q}(y)[x]$. Then $f_k(x,y)$ is a polynomial in $x$ of the form 
\[
f_k(x,y) = \sum_{i=1}^{k} f_{k,i}(y) x^{k-i}.
\]
where the coefficients are given by the recurrence relation
\[
f_{k,i}(y) = \dfrac{-y}{y-1} \sum_{j=1}^{i-1} f_{k,j}(y) \binom{k-j}{i-j}.
\]
\end{lemma}
\begin{proof}
Since the degree of $y f_k(x+1,y) - f_k(x,y)$ in $x$ is $k-1$, $f_k(x,y)$ must also have degree $k-1$. So we may write  
\[
f_k(x,y) = \sum_{i=1}^{k} f_{k,i}(y) x^{k-i}.
\]
Then we have that
\begin{align*}
&y f_k(x+1,y) - f_k(x,y) = h(k) x^{k-1} \implies \\
&y \sum_{i=1}^{k} f_{k,i}(y) \sum_{j=0}^{k-i} x^{k-i-j} - \sum_{i=1}^{k} f_{k,i}(y) x^{k-i} = h(k) x^{k-1}
\end{align*}
from which we easily see that $f_{k,1}(y) = h(k)/(y-1)$. Also by equating the coefficient of $x^{k-i}$ to 0 for $k \geq i \geq 2$ we obtain the recurrence relation 
\[
f_{k,i}(y) = \dfrac{-y}{y-1} \sum_{j=1}^{i-1} f_{k,j}(y) \binom{k-j}{i-j}.
\]
\end{proof}

The recurrence relation
\[
f_{k,i}(y) = \dfrac{-y}{y-1} \sum_{j=1}^{i-1} f_{k,j}(y) \binom{k-j}{i-j}, \,\,\,\, f_{k,1}(y) = h(k)/(y-1)
\]
allows us to derive some basic properties of $f_{k,i}(y)$. First we can see that $f_{k,i}$ is uniquely determined by $h(k)$. It can be shown inductively that each $f_{k,i}(y) \in \mathbb{Q}(y)$ is of the form
\[
f_{k,i} = \dfrac{\phi_{k,i}(y)}{(y-1)^i}
\]
where $\phi_{k,i} \in \mathbb{Q}[y]$ is a polynomial of degree at most $i-1$. Moreover $\phi_{k,i}(0) = 0$ for $i \geq 2$, $\phi_{k,1}(0)= h(k)$. Also if $h(k) \in \mathbb{Z}$ then $\phi_{k,i} \in \mathbb{Z}[y]$, and different choices of $h(k)$ only give multiples of $f_k(x,y)$ and so of $\mu_{k,z}$. Thus we have the following corollary about the uniqueness of $p$-adic distributions defined by functions $f_k(x,y)$.


\begin{cor}
\label{h(k)uniqueness}
In addition to the hypotheses of Theorem \ref{prep}, also assume that $f_k(x,y) \in \mathbb{Q}(y)[x]$. Then $f_k(x,y)$ and so $\mu_{k,z}$ is uniquely determined by the choice of $h(k)$ which only depends on $k$. 
\end{cor}

Up to now we proved the uniqueness of the $p$-adic distributions of the form 
\[
\mu_{k,z}(a+(p^N)) =p^{N(k-1)} z^a f_k(a/p^N, z^{p^N}) 
\]
for $f_k(x,y) \in \mathbb{Q}(y)[x]$ defined at $(0,0)$. In the next section we will show the existence of such distributions by giving an explicit example, and indeed we will see that the given $f_k(x,y)$ is a very familiar function. Here we will study further properties of $\mu_{k,z}$. First we will show that $\mu_{k,z}$ is bounded and so extends to a $p$-adic measure.

\begin{thm}
\label{generalmeasure}
Let $z \in C_p$ with $|z-1|_p \geq 1$. Let $f_k(x,y) \in \mathbb{Q}(y)[x]$ be as above. Then the $p$-adic distribution
\[
\mu_{k,z} (a + (p^N)) = p^{N(k-1)} z^a f_k(a/p^N, z^{p^N})
\]
extends to a $p$-adic measure on $\mathbb{Z}_p$. Also if $h(k) \in \mathbb{Z}$ then $|\mu_{k,z}(a+ (p^N))|_p \leq 1$ for any $N$ and $a$. 
\end{thm}

\begin{proof}
We need to show that $\mu_{k,z}$ is bounded on compact open subsets of $\mathbb{Z}_p$ and it is enough to check boundedness for subsets of the form $a+(p^N)$. Again we drop the index $z$ whenever it is understood to simplify the notation. By Corollary \ref{h(k)uniqueness} it is enough to prove the claim only for one choice of $h(k)$, so we may assume that $h(k) \in \mathbb{Z}$. 

By Lemma \ref{coefficientrecursion} and the discussion following it, we have that
\[
f_k(x,y) = \sum_{i=1}^{k} f_{k,i}(y) x^{k-i}
\]
where $f_{k,i}(y) = \phi_{k,i}(y)/(y-1)^i$ for some $\phi_{k,i}(y) \in \mathbb{Z}[y]$ with deg$(\phi_{k,i}(y)) = i-1$. Then
\[
\mu(a + (p^N)) = p^{N(k-1)} z^a \sum_{i=1}^{k} f_{k,i}(z^{p^N}) \dfrac{a^{k-i}}{(p^N)^{k-i}} = z^a \sum_{i=1}^{k} f_{k,i}(z^{p^N}) a^{k-i} p^{N(i-1)}.
\]
We have that $|a^{k-i} p^{N(i-1)}|_p \leq 1$ for $i=1,2,...,k$. So we need to show that $|z^a f_{k,i}(z^{p^N})|_p \leq 1$. 

If $|z|_p \leq 1$ then $|\phi_{k,i}(z^{p^N})|_p, |z^a|_p \leq 1$ (Note that $\phi_{k,i}(z^{p^N}) \in \mathbb{Z}[z^{p^N}]$). It follows that $|z^a f_{k,i}(z^{p^N})|_p \leq 1$. 

Now assume that $|z|_p > 1$. Since by assumption $|z^{p^N}-1|_p \geq 1$ we must have $|z^{p^N}|_p = |z^{p^N}-1|_p > 1$. Also $|z^a|_p \leq |z^{p^N}|_p$. Since deg$(\phi_{k,i}(y)) = i-1$ we have that $|z^a \phi_{k,i}(z^{p^N})|_p \leq {|z^{p^N}|^{i}}_p$ and so $|z^a f_{k,i}(z^{p^N})|_p \leq 1$ as desired. 
\end{proof}

\begin{rmk}
\normalfont{The $p$-adic measures $\mu_{k,z}$ defined above and the Bernoulli distributions have similar properties. It is easy to show that $B_k(x)$ is the only polynomial in one variable that can be used to define a $p$-adic distribution in this way. The main difference between $\mu_{k,z}$ and $\mu_{B,k}$ is that $\mu_{k,z}$ is bounded so extends to a measure contrary to $\mu_{B,k}$. The reason behind this difference is that $B_k(x)$ is a polynomial of degree $k$ but deg$_x(f_k(x,y))$ is $k-1$.}
\end{rmk}

We end this section by evaluating some integrals with respect to $\mu_{k,z}$. We will use these results in the next section. We fix $z$ with $|z-1|_p \geq 1$ and may denote $\mu_{k,z}$ by $\mu_k$ if there is no confusion.

\begin{prop}
\label{overZ_p}
With the above notation we have
\begin{align*}
i)& \int\limits_{\mathbb{Z}_p} 1 d\mu_{k} = f_k(0,z), \\ \nonumber
ii)& \int\limits_{\mathbb{Z}^{*}_p} 1 d\mu_{k,z} = f_k(0,z) - p^{k-1} f_k(0,z^p).
\end{align*}
\end{prop}

\begin{proof}
It follows by  Lemma \ref{fdecomposition} and the definition of these integrals.
\end{proof}



\begin{prop}
\label{indexshift}
Let $r$ and $k \geq 1$ be integers. Then
\begin{align*}
i)& \int\limits_{\mathbb{Z}_p} x^r d\mu_{k} = \dfrac{h(k)}{h(1)} \int\limits_{\mathbb{Z}_p} x^{r+k-1} d\mu_{1}, r \geq 0\\
ii) & \int\limits_{\mathbb{Z}_p^{*}} x^r d\mu_{k} = \dfrac{h(k)}{h(1)} \int\limits_{\mathbb{Z}_p^{*}} x^{r+k-1} d\mu_{1}.
\end{align*}
\end{prop}

\begin{proof}
By definiton
\begin{eqnarray}
\int\limits_{\mathbb{Z}_p} x^r d\mu_{k} &=& \lim_{N \rightarrow \infty} \sum_{a=0}^{p^N - 1} a^r p^{N(k-1)} z^a f_k(a/p^N, z^{p^N}) \nonumber \\ &=& \lim_{N \rightarrow \infty} \sum_{a=0}^{p^N - 1} a^r z^a \sum_{i=1}^{k} f_{k,i}(z^{p^N}) a^{k-i} p^{N(i-1)} \nonumber
\end{eqnarray}

By the proof Theorem \ref{generalmeasure} we have $\left|f_{k,i}(z^{p^N}) a^{r+k-i} z^a p^{N(i-1)}\right|_p \leq 1/p^{N(i-1)}$. Thus the terms corresponding to $i\geq 2$ in the above sum tend to zero as $N \rightarrow \infty$. So it is enough to take the sum only for $i=1$. Then since $f_{k,1} = h(k) / (y-1)$ and $f_1(x,y) = f_{1,1} (y)=h(1)/(y-1)$ we obtain 
\begin{eqnarray}
\int\limits_{\mathbb{Z}_p} x^r d\mu_{k} = \lim_{N \rightarrow \infty} f_{k,1}(z^{p^N}) \sum_{a=0}^{p^N - 1} a^{r+k-1} z^a = \dfrac{h(k)}{h(1)}\int\limits_{\mathbb{Z}_p} x^{r+k-1} d\mu_{1} \nonumber
\end{eqnarray}
which proves the first part. 

For the second part,  it is enough to change the summation by $\sum_{a=0,\, p\nmid a}^{p^N - 1}$ and use the fact that $|a^r|_p =1$ for any $a$ not divisible by $p$. 
\end{proof}


Choosing $r=0$ in Proposition \ref{indexshift} we see that
\[
\int\limits_{X} 1 d\mu_{k} = \dfrac{h(k)}{h(1)} \int\limits_{X} x^{k-1} d\mu_{1}
\]
where $X = \mathbb{Z}_p$ or $\mathbb{Z}_p^{*}$. It is reasonable to choose $h(k)$ so that we have the equality
\[
\int\limits_{X} 1 d\mu_{k} = k \int\limits_{X} x^{k-1} d\mu_{1}
\]
as we have for the Bernoulli measures. In the next section we will show the existence of $p$-adic distributions of this form satisfying this identity explicitly. So from now on we may take $h(k) = k$. 


\section{Twisted Bernoulli measures}
\label{tbnumbers}
First we recall the twisted Bernoulli polynomials which were defined in the Introduction. We define the twisted Bernoulli polynomials $\beta_{k}(x, y)$ for $k=0,1,2....$ by 

\[
\dfrac{t e^{xt}}{y e^t - 1} = \sum\limits_{k=0}^{\infty} \dfrac{\beta_{k}(x, y)}{k!} t^k
\]
where $y \neq 1$. For $y=1$, we may set $\beta_k(x,1) = B_k(x)$, where $B_k(x)$ is the $k$-th Bernoulli polynomial. But unless otherwise specified we always omit the case $y=1$. We see that $\beta_0=0$. For $x=0$, $\beta_k(0,y)$ is called as the $k$-th twisted Bernoulli number and may be denoted by $\beta_k(0,y)=\beta_k(y)$.

First we give some basic properties of $\beta_k(x,y)$ that we will need later. These follow by direct computation. The reader may also refer to \cite{apostol} for details. We may use these identities without any explanation and further reference. 

First we have that $\beta_0=0$. Let $\beta_k(y) = \beta_k(0,y)$. Then we see that 
\[
\beta_k(x,y) = \sum_{i=0}^{k} \binom{k}{i} \beta_i(y) x^{k-i}.
\]
We also have that for $k\geq 2$
\[
\beta_k(y) = y \sum_{i=0}^{k} \binom{k}{i} \beta_i(y) \mbox{  or equivalently  } \beta_k(y) = \dfrac{y}{1-y} \sum_{i=0}^{k-1} \binom{k}{i} \beta_i(y)
\]

Final equation that we may need is 
\[
\beta_k(a+b, y) = \sum_{i=0}^{k} \binom{k}{i} \beta_i(a,y) b^{k-i}.
\]
As $\beta_0=0$ we may remove the index $i=0$ in these identities whenever we need. 

Now we will show that $\beta_k(x,y)$ satisfy the hypothesis of in Theorem \ref{prep} and so provide the first nontrivial example of $p$-adic measures discussed here. They additionally play the role of $d(x^k)$ as discussed in the previous section. 

\begin{thm}
Let $\beta_k(x,y)$ be the $k$-th twisted Bernoulli polynomial. Then the map defined as 
\[
\mu_{k,z}(a + (p^N)) = p^{N(k-1)} z^a \beta_k(a/p^N, z^{p^N})
\]
extends to a $p$-adic measure. Moreover $y \beta_k(x+1,y) - \beta(x,y) = k x^{k-1}$, i.e. with the notation of the previous section $h(k) = k$.
\end{thm}

\begin{proof}
By Lemma \ref{fdecomposition} and the notes after it, in order to prove that $\mu_{k,z}$ is a $p$-adic distribution it is enough to show the equality
\[
\beta_k(a,y) = p^{k-1} \sum_{b=0}^{p-1} y^b \beta_k((a+b)/p, y^p).
\]
Indeed we shall prove that for any positive integer $M$ the following equality holds;
\[
\beta_k(a,y) = M^{k-1} \sum_{b=0}^{M-1} y^b \beta_k((a+b)/M, y^M).
\]
Now by definition of the twisted Bernoulli numbers we have

\begin{eqnarray}
\sum_{k=0}^{\infty} \sum_{b=0}^{M-1} \beta_k((a+b)/M,y^M) y^b \dfrac{t^k}{k!} &=& \sum_{b=0}^{M-1} \dfrac{ty^b e^{t(a+b)/M }}{y^M e^t - 1} = \dfrac{t e^{ta/M}}{y^M e^t - 1} \sum_{b=0}^{M-1} (y e^{t/M})^b \nonumber \\ 
&=& \dfrac{t e^{ta/M}}{y^M e^t - 1} \dfrac{y^M e^t - 1}{y e^{t/M} - 1} = M \dfrac{(t/M) e^{a(t/M)}}{y e^{t/M} - 1} \nonumber \\
&=& M \sum_{k=0}^{\infty} \beta_k(a,y) \dfrac{t^k}{M^k k!}. \nonumber
\end{eqnarray} 
By equating the coefficients of $t^k$ we obtain 
\[
\beta_k(a,y) = M^{k-1} \sum_{b=0}^{M-1} y^b \beta_k((a+b)/M, y^M)
\]
which implies that $\mu$ extends to a $p$-adic distribution. Since $\beta_k(x,y)$ satisfy the hypothesis of Theorem \ref{generalmeasure}, $\mu$ indeed extends to a $p$-adic measure. 

By direct computation we see that $y \beta_k(x+1,y) - \beta(x,y) = k x^{k-1}$, or we may refer to \cite{apostol}. 
\end{proof}

\begin{defn}
Let $z \in \mathbb{C}_p$ such that $|z-1|_p \geq 1$ and $k \in \mathbb{Z}_{\geq 1}$. We define the $k$-th twisted Bernoulli measure on $\mathbb{Z}_p$ as 
\[
\mu_{\beta, k,z} (a + (p^N) ) = p^{N(k-1)} z^a \beta_k(a/p^N, z^{p^N})
\] 
on the subsets of the form $a+(p^N)$ and extend it to all compact open subsets of $\mathbb{Z}_p$ linearly. 
\end{defn} 

Now we have a concrete example of a sequence of $p$-adic distributions $\mu_{k,z}$ of the form given in Theorem \ref{prep}. By the uniqueness of such measures (up to a  multiple), many of the results of this section may be generalized to other measures in this form. As we will see, working with $\beta_k$ instead of an arbitrary $f_k$ may be easier to derive some results. So in this section we will only work with twisted Bernoulli measures. Accordingly to simplify the notation we may denote them by $\mu_{k,z}$. Also we may use $\mu_k$, $\mu_z$ or only $\mu$ if there is no confusion. 

\begin{rmk}
\normalfont{
In the notation of Section \ref{general}, we have that 
\[
\beta_k(x,y) = f_k(x,y) = \sum_{i=1}^{k} f_{k,i}(y) x^{k-i}  \implies f_{k,i} = \binom{k}{i} \beta_i(y).\,\,\,
\]}
\end{rmk}

\begin{ex}\normalfont{
We compute $\mu_{z,1}$ and $\mu_{z,2}$. Since in general
\[
\beta_n(a,y) = \sum_{k=0}^{n} \binom nk\beta_k(y) a^{n-k}
\]
 we have that $\beta_1(a,y) = \beta_1(y) = 1 / (y - 1)$. So 
\[
\mu_{z,1} (a + (p^N)) = z^a / (z^{p^N} - 1)
\]
which is the same measure up to a minus sign defined by Koblitz in \cite{koblitz_book1}. 

Similarly for $n=2$ we have
\[
\beta_2(a,y) = 2\beta_1(y) a + \beta_2(y) = \dfrac{2a}{y - 1} + \dfrac{-2 y}{(y - 1)^2},
\]
and so 
\[
\mu_{z,2}(a + (p^N)) = z^a \left(\dfrac{a}{z^{p^N} - 1} + \dfrac{-  p^{N} z^{p^N}}{(z^{p^N} - 1)^2} \right).
\]}
\end{ex}

We need a final identity about twisted Bernoulli polynomials for later use. 

\begin{lemma}
\label{symmetry}
For any $k \geq 0$, $a$ and $y\neq 0,1$, the following equality holds;
\[
(-1)^k \beta_k(a,1/y) = y \beta_k(1-a,y).
\]
\end{lemma}

\begin{proof}
By definition we have that
\[
\sum_{k=0}^{\infty} \beta_k(a,1/y) (-1)^k \dfrac{t^k}{k!} = \dfrac{(-t)e^{a(-t)}}{(1/y) e^{-t} - 1}.
\]
We may manipulate the right hand side as
\[
\dfrac{(-t)e^{a(-t)}}{(1/y) e^{-t} - 1} = \dfrac{-y t e^{-at}}{e^{-t}(1-y e^t)} = \dfrac{y t e^{(1-a)t}}{y e^t - 1} = y \sum_{k=0}^{\infty} \beta_k(1-a, y) \dfrac{t^k}{k!}. 
\]
Equating the coefficients of $t^k$ in both expansions we obtain 
\[
(-1)^k \beta_k(a,1/y) = y \beta_k(1-a,k)
\]
as desired.
\end{proof}

Now we will prove some results about $p$-adic integration using twisted Bernoulli measures. First we adapt the notation of Section \ref{general} to twisted Bernoulli measures. The following corollary follows by definition and Propositions \ref{overZ_p} and \ref{indexshift}.
\begin{cor}
Let $z\in C_p$ with $|1-z|_p \geq 1$. Let $\mu_{\beta,k,z} = \mu_{k,z}$ be as above. Then
\begin{itemize}
\item[i)] $\displaystyle  k \int\limits_{\mathbb{Z}_p} x^{k-1} d\mu_{1,z} = \int\limits_{\mathbb{Z}_p} 1 d\mu_{k,z} = \beta_k(z)$
\item[ii)] $\displaystyle k \int\limits_{\mathbb{Z}^{*}_p} x^{k-1} d\mu_{1,z} =  \int\limits_{\mathbb{Z}^{*}_p} 1 d\mu_{k,z} = \beta_k(z) - p^{k-1} \beta_k(z^p)$
\item[iii)] $\displaystyle \int\limits_{\mathbb{Z}_p} x^r d\mu_{k,z} = k \int\limits_{\mathbb{Z}_p} x^{r+k-1} d\mu_{1,z}$  \normalfont{for} $r \geq 0$
\item[iv)]$\displaystyle \int\limits_{\mathbb{Z}_p^{*}} x^r d\mu_{k} = k \int\limits_{\mathbb{Z}_p^{*}} x^{r+k-1} d\mu_{1}$.
\end{itemize}
\end{cor}

The following theorem relates the twisted Bernoulli measures $\mu_{k,z}$ and $\mu_{k,1/z}$. It was first proved in \cite{coleman}. Here we give a different proof. 

\begin{thm}
Let $z \in \mathbb{C}_p - \{0\}$ with $|z-1|_p \geq 1$. Then 
\begin{align*}
& \int\limits_{\mathbb{Z}_p} 1 d\mu_{k,1/z} = (-1)^k \int\limits_{\mathbb{Z}_p} 1 d\mu_{k,z},  k \geq 2 \nonumber \\ 
& \int\limits_{\mathbb{Z}_p} 1 d\mu_{1,1/z} = - 1 - \int\limits_{\mathbb{Z}_p} 1 d\mu_{1,z} 
\end{align*}
\end{thm}

\begin{proof}
We just use the definition of $p$-adic integration and Lemma \ref{symmetry} 
\begin{eqnarray}
\int\limits_{\mathbb{Z}_p} 1 d\mu_{k,1/z} &=& \lim_{N \rightarrow \infty} \sum_{a=0}^{p^N - 1} p^{N(k-1)} z^{-a} \beta_k(a/p^N, 1/z^{p^N}) \nonumber \\ &=& \lim_{N \rightarrow \infty} (-1)^k \sum_{a=0}^{p^N - 1} p^{N(k-1)} z^{-a} z^{p^N} \beta_k(1 - a/p^N, z^{p^N}) \nonumber \\
&=& (-1)^k \lim_{N \rightarrow \infty} \sum_{a=0}^{p^N - 1} p^{N(k-1)} z^{p^N - a} \beta_k((p^N - a) / p^N , z^{p^N}) \nonumber \\
&=& (-1)^k \lim_{N \rightarrow \infty} \sum_{b=1}^{p^N} p^{N(k-1)} z^b \beta_k(b/p^N, z^{p^N}) \nonumber \\
&=& (-1)^k \lim_{N \rightarrow \infty} p^{N(k-1)} \left[\sum_{b=0}^{p^{N-1}} z^b \beta_k(b/p^N, z^{p^N}) + z^{p^N} \beta_k(1, z^{p^N}) - \beta_k(0, z^{p^N})\right] \nonumber
\end{eqnarray}
Now for $k \geq 2$ we have $z^{p^N} \beta_k(1, z^{p^N}) = \beta_k(0,z^{p^N})$ and so  
\[
\int\limits_{\mathbb{Z}_p} 1 d\mu_{k,1/z} = (-1)^k \int\limits_{\mathbb{Z}_p} 1 d\mu_{k,z}.
\]
Similarly $z^{p^N} \beta_1(1, z^{p^N}) = 1 + \beta_1(0,z^{p^N})$ which implies that
\[
\int\limits_{\mathbb{Z}_p} 1 d\mu_{1,1/z} = - 1 - \int\limits_{\mathbb{Z}_p} 1 d\mu_{1,z}
\]
as desired. 
\end{proof}

Since $ \int\limits_{\mathbb{Z}^{*}_p} 1 d\mu_{k,z} = \int\limits_{\mathbb{Z}_p} 1 d\mu_{k,z} -  p^{k-1} \int\limits_{\mathbb{Z}_p} 1 d\mu_{k,z^p}$ the following result directly follows. 

\begin{cor}
Let $z \in C_p - \{0\}$ with $|z-1|_p \geq 1$. Then for any $k \geq 1$,
\[
\int\limits_{\mathbb{Z}^{*}_p} 1 d\mu_{k,1/z} = (-1)^k \int\limits_{\mathbb{Z}^{*}_p} 1 d\mu_{k,z}
\]
\end{cor}

To finalize we give the relation between Bernoulli measures and twisted Bernoulli measures which gives an alternative proof of the fact that the regularized Bernoulli distributions are bounded and so extend to $p$-adic measures. 

\begin{thm}
Let $c \in \mathbb{Z}_{\geq 1}$ with $p \nmid c$ and let $\mu_c$ denote the $c$-th root of unities. Then 
\[
c^{-k} \sum_{\zeta \in \mu_c - \{1\}} \mu_{k,z}(a + (p^N)) = - \mu_{B, k, c^{-1}}(a+(p^N)).
\]
Here for $\alpha \in {\mathbb{Z}_p^{*}}$, $\mu_{B,k,\alpha}$ is the Bernoulli measure defined as \\
$\mu_{B, k, \alpha} (a+(p^N)) = \mu_{B,k} \left(a + (p^N)\right) - {\alpha}^{-k} \mu_{B,k}(\{a \alpha\}_N + (p^N))$ where $\{a \alpha\}_N$ denotes the residue class of $a \alpha$ mod $(p^N)$. 
\end{thm}

\begin{proof}
We set $\beta_k(x,1) = B_k(x)$. Now
\begin{align*}
\sum_{k=0}^{\infty} \sum_{\zeta \in \mu_c} {\zeta}^a \beta_k(a/p^N, {\zeta}^{p^N}) \dfrac{t^k}{k!} &= \sum_{\zeta \in \mu_c} \dfrac{t e^{at/p^N} {\zeta}^a}{{\zeta}^{p^N}e^t -1} = -t e^{at/p^N} \sum_{\zeta \in \mu_c} \dfrac{{\zeta}^a}{1 - {\zeta}^{p^N}e^t} \\ 
&= -t e^{at/p^N} \sum_{\zeta \in \mu_c} {\zeta}^a \sum_{l=0}^{\infty} {({\zeta}^{p^N}e^t)}^l = -t \sum_{l=0}^{\infty} e^{t(a+l p^N)/p^N} \sum_{\zeta \in \mu_c} {\zeta}^{a+l p^N} \\ 
&= -ct  \sum_{l=0, c \mid (a+lp^N)}^{\infty}  e^{t(a+l p^N)/p^N}
\end{align*}
Now we manipulate the sum $\sum_{l=0, c \mid (a+lp^N)}^{\infty}  e^{t(a+l p^N)/p^N}$. We have
\begin{align*}
c \mid (a+lp^N), \,\, l=0,1,2... \iff l = a_N + sc,\,\, s=0,1,2,...
\end{align*}
where $a_N$ is the unique integer with  $0 \leq a_N \leq c-1$ and $p^{N} a_N \equiv -a$ (mod $c$). Then we have
\[
\sum_{l=0, c \mid (a+lp^N)}^{\infty}  e^{t(a+l p^N)/p^N} = e^{t (a + p^N a_N)/p^N} \sum_{s=0}^{\infty} e^{(ct)s}
\]
which implies that 
\begin{align*}
\sum_{k=0}^{\infty} \sum_{\zeta \in \mu_c} {\zeta}^a \beta_k(a/p^N, {\zeta}^{p^N}) \dfrac{t^k}{k!} &= -ct e^{t (a + p^N a_N)/p^N} \sum_{s=0}^{\infty} e^{(ct)s} \\
&= \dfrac{ct}{e^{ct}-1} e^{ct(a+p^N a_N )/(cp^N)}  \\ 
&= \sum_{k=0}^{\infty} B_k\left(\dfrac{a+p^N a_N}{c p^N}\right) \dfrac{c^k t^k}{k !}.
\end{align*}

Taking out the coefficient of $t^k/k!$ and multiplying it by $p^{N(k-1)}$ we obtain 
\begin{align*}
\sum_{\zeta \in \mu_c - \{1\}} \mu_{k,z}(a + (p^N)) & = c^k \mu_{B,k} \left(\dfrac{a+p^N a_N}{c} + (p^N)\right) - \mu_{B,k}(a + (p^N)) \\
& = c^k \mu_{B,k} \left(\dfrac{a}{c} + (p^N)\right) - \mu_{B,k}(a + (p^N)).
\end{align*}
The last equality follows from the fact that $p^N a_N / c \in (p^N)$. Multiplying both sides by $c^{-k}$ we obtain the desired equality. 
\end{proof}

\end{document}